\documentclass[11pt,reqno]{amsart}
\usepackage{amsmath}
\usepackage{amssymb}
\usepackage[bookmarks]{hyperref}
\usepackage{enumerate}
\usepackage{tikz}
\usepackage{tcolorbox}

\newtheorem{theorem}{Theorem}[section]
\newtheorem{lemma}[theorem]{Lemma}
\newtheorem{corollary}[theorem]{Corollary}
\newtheorem{proposition}[theorem]{Proposition}
\newtheorem{prop}[theorem]{Proposition}

\theoremstyle{definition}
\newtheorem{remark}[theorem]{Remark}
\newtheorem{definition}[theorem]{Definition}
\newtheorem{problem}[theorem]{Problem}
\newtheorem{question}[theorem]{Question}
\newtheorem{example}[theorem]{Example}

\numberwithin{equation}{section}

\newcommand{\cM}{\mathcal{M}}

\newcommand{\h}{\mathcal{H}}
\newcommand{\D}{\mathbb{D}}
\newcommand{\F}{\mathcal{F}}
\newcommand{\C}{\mathbb{C}}

\newcommand{\supp}{\operatorname{supp}}

\newcommand{\cls}{\operatorname{cl}}
\newcommand{\spn}{\operatorname{span}}

\begin{document}

\allowdisplaybreaks[3]

\title{Forward Operator Monoids}


\author{Christopher Felder}
\address{Department of Mathematics, Indiana University, Bloomington, Indiana, 47405}
\email{cfelder@iu.edu}

\makeatletter
\@namedef{subjclassname@2020}{%
  \textup{2020} Mathematics Subject Classification}
\makeatother

\date{\today}



\begin{abstract}
This work studies collections of Hilbert space operators which possess a strict monoid structure under composition. These collections can be thought of as discrete unital semigroups for which no subset of the collection is closed under composition (apart from the trivial subset containing only the identity). We call these collections \textit{forward operator monoids}. Although not traditionally painted in this light, monoids of this flavor appear in many areas of analysis; as we shall see, they are intimately linked to several well-known problems. The main aim of this work is to study three classes of vectors associated to a given operator monoid: cyclic vectors, generalized inner vectors, and vectors which are both cyclic and inner (we refer to this last class as aleph vectors). We show, when they exist, aleph vectors are unique up to a multiplicative constant.  We also show, under the lens of a least-squares  problem, how an aleph vector can be used to characterize all cyclic and inner vectors. As an application of this theory, we consider monoids related to the Periodic Dilation Completeness Problem and the Riemann Hypothesis.  After recovering some known results, we conclude by giving a further reformulation of the Riemann Hypothesis based on work of B\'{a}ez-Duarte \cite{B-D} and Noor \cite{Noor}. 
\end{abstract}

\maketitle


\section{Introduction and Background}

In many areas of mathematics, it is often of interest or importance to understand the dynamical behavior of a mapping applied to a point in its domain: given a set $X$ and a mapping $T: X \to X$, what can be said as $T$ is repetitively applied to an element $x \in X$? How `far-traveling' is $x$ under the iterative action of $T$? If $X$ is a topological space, what can be said of the density of the orbit of $x$ under $T$? If $X$ is a vector space, what can be said of the linear combinations of $T^nx$ as $n$ is exhausted ad infinitum? 

This work aims to explore questions of this ilk when the mappings are bounded linear transformations on a Hilbert space. We will consider not only the iterates of a single transformation, but also the action of families of linear operators which adhere to certain algebraic properties. We call these families \textit{forward operator monoids}.

In short, these collections of operators can be thought of as discrete unital semigroups for which no finite subset of the collection is closed under composition (apart from the trivial subset containing only the identity operator). 
The most elementary example of such an operator monoid is $\{ T^k \}_{k = 0}^\infty$, where $T$ is a linear operator acting on a Hilbert space, and $T^j \neq T^k$ for any $j,k \ge 0$ whenever $j \neq k$. The monoid property enjoyed by the operators here is inherited from the `strict' monoid $\mathbb{N} = \{ 0, 1, 2, \ldots \}$ with the binary operation of addition (in this setting we have $T^jT^k = T^{j+k}$ for all $j, k \in \mathbb{N}$). This is an important example, but the most interesting examples are collections of operators which follow a multiplicative rule. We discuss several such collections in Section \ref{apps}.

Although these collections are sometimes referred to as \textit{semigroups}, their structure is often much richer than this name implies.  
While they are not often presented in the way which we shall see here, operator monoids of this type arise in many areas of mathematics. As we will see, they are closely connected to several well-known problems, including the Periodic Dilation Completeness Problem and the Riemann Hypothesis.
It is not our goal to unilaterally tackle such problems, but rather to work within their periphery, addressing some interesting problems and unifying some of the theory.

Our main objective is to study three classes of vectors associated to a given operator monoid: 
\begin{enumerate}
\item \underline{Cyclic vectors}: vectors for which the linear combinations of the operator monoid applied to the vector are dense in the space at hand.
\item \underline{Inner vectors}: vectors which are orthogonal to any (non-identity) element of the monoid applied to itself; as we will see, these vectors are antipodal, in some sense, to cyclic vectors. 
\item \underline{Aleph vectors}: vectors which are both cyclic and inner. 
\end{enumerate}

In order to expedite more enlightening conversation surrounding these objects, we keep the forthcoming discussion of background to a minimum, instead providing more relevant detail, when appropriate, in subsequent sections. 

For example, the notion of cyclicity has wide-ranging relevance in many areas of mathematics, but we will not indulge in reviewing that vast literature here. In lieu of this, we will mention works which contain key ideas and point to references therein, as they are relevant to our discourse.
In particular, Section \ref{apps} discusses several concrete examples for which cyclic vectors are relevant to this work (important literature here, which we discuss later, includes \cite{B-D, MR2226127, Nik, Noor, MR4203042}).

The history of inner vectors begins in a factorization result for functions in the Hardy space of the unit disk, going back to F. Riesz \cite{MR1544621}; these inner factors are functions that are orthogonal to any non-zero iterate of the forward shift applied to the function itself. This notion has been naturally extended to many other spaces of analytic functions on the disk and other domains, to function spaces of several complex variables (wherever the coordinate shift operators are bounded), to non-commutative function spaces, and to general reproducing kernel Hilbert spaces. We point to \cite{MR4549696} and the references therein, and also to the recent survey \cite{Sampat}.

Despite classical standing, it seems that inner vectors in a most generalized operator-theoretic sense have not been considered outside of the works of Cheng, Mashreghi, and Ross. In \cite{MR4061942}, these authors explore a notion of inner  vectors for Toeplitz operators acting on the Hardy space of the disk. In \cite{MR4019469}, the same authors push further by entertaining inner vectors for general Hilbert space operators. In both of these works, a vector $h$ in a Hilbert space $\h$ is inner for a bounded linear operator $T$ acting on $\h$ if $\langle h, T^k h \rangle_\h = 0$ for all $k \ge 1$.

Aleph vectors appear to be unstudied, with the definition of such vectors novel to this work. Although it will take some effort to uncover the nature of these vectors, it is our hope to convince the reader that their study is worthwhile, as they encode information about both cyclic and inner vectors in general.

The outline and contributions of this paper are as follows:
\begin{itemize}
\item Section \ref{n-and-d} presents relevant definitions and background, including the formal definition of a forward operator monoid, and vectors which are cyclic, inner, and aleph for these operator monoids. 
\item Section \ref{aleph} uncovers some basic properties of aleph vectors, including their existence and uniqueness.
\item Section \ref{least-squares} analyzes a least-squares approximation problem which links aleph vectors to the characterization of both cyclic and inner vectors (see Theorems \ref{al-approx} and \ref{opa}, respectively). 
\item Section \ref{apps} applies the results of the previous sections to some concrete examples, including forward operator monoids related to the Periodic Dilation Completeness Problem and the Riemann Hypothesis. After recovering some known results, Theorem \ref{RH-reform} gives a reformulation of the Riemann Hypothesis, in the spirit of the least-squares approximation discussed in Section \ref{least-squares}.
\end{itemize}


\section{Notation and Definitions}\label{n-and-d}
Throughout this work, $\h$ will be a separable infinite-dimensional Hilbert space and $\mathcal{B}(\h)$ will denote the set of bounded linear operators on $\h$. We will consider collections $\F \subseteq \mathcal{B}(\h)$ which are discrete, contain the identity, and have strict monoid structures. 
The \textit{strict} monoid structure means that $\F$ can be indexed by a monoid which does not possess inverses, and where composition of elements in $\F$ respects the binary operation on the indexing monoid.  We denote such an indexing monoid by the triple $(\Lambda, b, n_0)$, where $\Lambda$ is a discrete set, $b$ is a commutative binary operator on $\Lambda \times \Lambda$, and $n_0 \in \Lambda$ is the identity element.
Let us precisely record this now:

\begin{definition}[Forward operator monoid]\label{FOM}
A collection $\F \subseteq \mathcal{B}(\h)$, containing the identity,  is a \textit{forward operator monoid} if 
there exists a strict monoid $(\Lambda, b, n_0)$, which is commutative and discrete, so that 
\[
\F = \{ T_j \}_{j\in \Lambda}
\]
and the map $\rho :  (\Lambda, b, n_0) \to (\F, \circ, I)$, given by 
\[
\rho(j) = T_j,
\]
is a monoid isomorphism (here, $\circ$ represents standard composition on $\mathcal{B}(\h)$ and $I$ is the identity on $\h$). 
\end{definition}

We note that the map $\rho$ above is akin to a `representation' of $\Lambda$ on $\h$. However, the finer properties of this mapping and the underlying monoid which indexes the collection of operators are not of interest here. The use of the this correspondence is solely to envelop, in one fell swoop, several different types of families of operators which are of interest. Essentially, these families possess the same basic algebraic structure but are indexed in different ways. 
One could, more awkwardly, refer to a forward operator monoid as a \textit{discrete unital operator semigroup for which no finite subset of the semigroup is closed under composition}; we prefer a more succinct terminology, and the above definition is born from this preference (in fact, our definition of forward operator monoid implies a bit more than this, but we feel the point is made). These collections should not be confused with continuous ($C_0$) semigroups or any other continuum of operators.

By definition, it is elementary to check that a forward operator monoid $\F = \{ T_j \}_{j\in \Lambda}$ possesses the following properties:
\begin{enumerate}
\item $T_{n_0} = I$, the identity on $\h$.
\item $T_j T_k = T_{b(j, k)}$ for all $j, k \in \Lambda$.
\item As $b$ is assumed commutative, the elements of $\F$ commute.
\item \label{forward1}$T_j T_k = I$ if and only if $j = k = n_0$.  
\item \label{forward2}$T_j T_k = T_k$ if and only if $j = n_0$.  
\end{enumerate}
Items (\ref{forward1}) and (\ref{forward2}) above will be referred to as the \textit{forward properties} of the operator monoid. These properties ensure that no element of $\F$ is nilpotent or periodic, and that composition of elements in $\F$ moves, in some sense, \textit{forward} through the monoid. Some of the results to come do not explicitly require this property, but we will point out when it is used.

The forward operator monoids with which we will primarily deal are additive monoids
\[
T_jT_k = T_{j+k} \ \ \forall j,k = 0, 1, 2, \ldots
\]
and multiplicative monoids
\[
T_jT_k = T_{jk} \ \ \forall j,k = 1, 2, 3, \ldots. 
\]

As previously mentioned, the most elementary constructive example of a forward operator monoid comes from iterating a single operator $T \in \mathcal{B}(\h)$ which is neither nilpotent nor periodic. The forward operator monoid is then $\F = \{ T^j \}_{j=0}^\infty$, where the monoid structure is simply inherited by the non-negative integers with addition. There are also natural forward operator monoids with a multiplicative structure, which we study in Section \ref{apps}.

In most cases, we will inquire about the behavior of a forward operator monoid acting on a single vector. Without explicit mention of the underlying set which indexes our operator monoid, we can make the following definitions, which play the leading part in this work. 
\begin{definition}\label{main-def}
Let $\F \subseteq \mathcal{B}(\h)$ be a forward operator monoid. We say that a non-zero vector $h \in \h$ is: 
\begin{enumerate}
\item  \underline{cyclic for $\F$}  (or $\F$-cyclic) if the set
\[
[h]_{\F}:= \cls_\h\left(\spn\{ T h : T \in \F \}\right)
\]
is equal to all of $\h$, where $\cls_\h(V)$ denotes the closure of the subspace $V \subseteq \h$. 
\item \underline{inner for $\F$} (or $\F$-inner) if
\[
\langle Th, h \rangle_\h = 0 \ \ \ \forall T \in  \F \setminus \{I\}.
\]
\item an \underline{aleph for $\F$} (or $\F$-aleph) if $h$ is both cyclic and inner for $\F$. 
\end{enumerate}
\end{definition}

Note for any $h\in \h$, the subspace $[h]_{\F}$ is always $\F$-invariant; if $u \in [h]_{\F}$, then $Tu \in [h]_{\F}$ for all $T \in \F$. Further, the subspace $[h]_{\F}$ is the smallest such subspace containing $h$. 
We will refer to $[h]_{\F}$ as the \textit{$\F$-invariant subspace generated by $h$}.

\begin{remark}
The nomenclature assigned here is not arbitrary; it is informed by the classical study of the forward shift on $\ell^2(\mathbb{N})$, given by
\[
(a_0, a_1, a_2,  \ldots) \mapsto (0, a_0, a_1, a_2, \ldots).
\] 
This operator is unitarily equivalent to the forward shift acting on Hardy space of the unit disk $\mathbb{D}$
\[
H^2:=\left\{f \in \operatorname{Hol}(\D) : \sup_{0\le r <1} \int_0^{2\pi} \left|f(re^{i\theta})\right|^2 \, d\theta < \infty \right\},
\]
where $\operatorname{Hol}(\D)$ is the collection of holomorphic functions on $\D$.
The forward shift $S$ in this setting is simply multiplication by the independent variable:
\[
(Sf)(z) = zf(z), \ \ \ z \in \D.
\]
If we consider the collection of iterates formed by $S$, say, $\mathcal{S} := \{ S^j \}_{j =0}^\infty$, it is elementary to check that this forms an additive forward operator monoid. Classically, a function $f \in H^2$ is called \textit{inner} if $|f| = 1$ a.e. on the unit circle $\mathbb{T}$. It turns out that this definition coincides with Definition \ref{main-def} when applied to the forward operator monoid $\mathcal{S}$. Inner and cyclic vectors for the shift have played a critical role in answering operator and function theoretic questions in $H^2$. 

For example, a celebrated result of Beurling describes all shift-invariant subspaces of $H^2$ as $\theta H^2$, where $\theta$ is an inner function for the monoid generated by the shift, and all cyclic vectors for the shift as so-called \textit{outer} functions, which also arise in the factorization theorem previously mentioned.

Further, it was shown in \cite{MR4692840} that for many Hilbert spaces of analytic functions on the disk, there is a unique aleph for the monoid generated by the shift; up to a multiplicative constant, this is the unit constant function $u(z) = 1$.
\end{remark}


\section{Aleph vectors}\label{aleph}

We begin by examining the existence and uniqueness of aleph vectors for a forward operator monoid. 
We will eventually show that when an aleph vector exists, it is unique (up to a multiplicative constant). 
However, in the following simple example, we immediately uncover that a forward operator monoid may not always possess an aleph vector, and in some cases may have neither a cyclic nor inner vector. 
\begin{example}
Let $\alpha \in \C$ be non-zero and not a root of unity.  Let $A = \alpha I$ and let $\mathcal{A}$ be the semigroup generated by the iterates of $A$, i.e. $\mathcal{A} = \{ A^j \}_{j=0}^\infty$.  For any $h \in \h$, we have $[h]_\mathcal{A} = \operatorname{span}\{h\}$, so no vector can be cyclic for $\mathcal{A}$. Further, for all $j\ge0$ and all non-zero $h \in \h$, we have
\[
\langle h, A^jh \rangle = \alpha^j \|h\|^2 \neq 0.
\]
In turn, $\mathcal{A}$ can have no inner vector. 
\end{example}

Nonetheless, as we will see in the next section, when an aleph vector does exist, it is of interest because it can be used, under the lens of a least-squares problem, to characterize both inner and cyclic vectors. 
Let us provide a lemma that will allow us to establish the uniqueness of aleph vectors.

\begin{lemma}\label{unit_ortho}
Let $\F \subseteq \mathcal{B}(\h)$ be a forward operator monoid and suppose that $u$ is an aleph vector for $\F$. For any $h \in \h$, have
\[
u \perp \cls_\h\left(\spn\left\{ Th : T \in \F \setminus \{ I\} \right\} \right). 
\]

\end{lemma}

\begin{proof}
Suppose that $\F$ is indexed by the monoid $(\Lambda, b, n_0)$ and suppose that $(F_m)$ is an increasing sequence of finite sets such that $F_m \to \Lambda$ (such a sequence exists as $\Lambda$ is assumed discrete). As $u$ is cyclic, for any $h\in \h$ there exist constants $(c_j)_{j \in \Lambda} \subseteq \C$ (depending on $h$) so that
\[
\sum_{j \in F_m} c_j T_j u \ \xrightarrow{m \to \infty} \ h.
\]
Further, for each $k \in \Lambda$, we have
\[
\sum_{j \in F_m} c_j T_{b(j, k)} u \ \xrightarrow{m \to \infty} \ T_k h.
\]
This yields
\[
\left\langle u, \sum_{j \in F_m} c_j T_{b(j, k)} u \right\rangle \ \xrightarrow{m \to \infty} \  \langle u, T_k h\rangle.
\]
Now, if we suppose that $k \neq n_0$, then it must be that $b(j, k) \neq n_0$. In turn, by the forward property of the operator monoid, $T_{b(j, k)} \neq I$. Thus, by the inner-ness of $u$, for all $k \neq n_0$, we have
\[
\left\langle u, \sum_{j \in F_m} c_j T_{b(j, k)} u \right\rangle = 0.
\]
Letting $m \to \infty$, we have $u \perp T_k h$ for all $k \neq n_0$. The result now follows, as $\spn\left\{T_k h :  k \in \Lambda\setminus \{ n_0\}\right\}$ is dense in $\cls_\h\left(\spn\left\{ Th : T \in \F \setminus \{ I\} \right\} \right)$.
\end{proof}

With this lemma in hand, we can immediately address the uniqueness of aleph vectors.
\begin{theorem}\label{uniqueness}
Let $\F \subseteq \mathcal{B}(\h)$ be a forward operator monoid. If an aleph vector for  $\F$ exists, then, up to a multiplicative constant, it is unique. 
\end{theorem}

\begin{proof}
Let $h \in \h$ and suppose that $u$ and $\mu$ are both aleph vectors for $\F$. WLOG, suppose $\|\mu\| = 1$ and put
\[
h = \langle h, \mu \rangle \mu + R,
\]
with $R \perp \mu$. 
Now notice that
\[
\langle h, u \rangle = \langle \langle h, \mu \rangle \mu + R, u \rangle
= \langle h, \mu \rangle \langle \mu, u \rangle + \langle R, u \rangle.
\] 
We claim that $\langle R, u \rangle = 0$. In order to see this, 
suppose $\F$ is indexed by the monoid $(\Lambda, b, n_0)$ and suppose that $(F_m)$ is an increasing sequence of finite sets such that $F_m \to \Lambda$. As $\mu$ is cyclic, there exist constants $(c_j)_{j \in \Lambda} \subseteq \C$ so that
\[
\sum_{j \in F_m} c_j T_j \mu \ \xrightarrow{m \to \infty} \ R.
\]
However, as $R \perp \mu$, the constant $c_{n_0}$ in the expression above must be zero. In turn, for each $m$, we have $\sum_{j \in F_m} c_j T_j \mu \in \spn\left\{ T\mu : T \in \F \setminus \{ I\} \right\}$ and by Lemma \ref{unit_ortho}, we have 
\[
\left\langle \sum_{j \in F_m} c_j T_j \mu, u \right\rangle = 0.
\]
Since
\[
0 = \left\langle \sum_{j \in F_m} c_j T_j \mu, u \right\rangle \ \xrightarrow{m \to \infty} \ \langle R, u \rangle, 
\]
we conclude that $\langle R, u \rangle = 0$. 

Thus, $\langle h, u \rangle =  \langle h, \mu \rangle \langle \mu, u \rangle$, and as $h$ was arbitrary, the result follows from a basic density argument. 

\end{proof}

Before moving on to establishing the usefulness of aleph vectors, let us end this section with a characterization of inner vectors.

\begin{prop}\label{cheng-inner}
Let $\F  \subseteq \mathcal{B}(\h)$ be a forward operator monoid. Let $h \in \h$ and let 
\[
P_\cM \colon  \h \to  \cM : =\cls_\h\left(\spn\left\{ Th : T \in \F \setminus \{ I\} \right\} \right)
\]
be the orthogonal projection from $\h$ onto $\cM$.

Then for any $h \in \h$, the vector $h - P_\cM h$ is $\F$-inner (or zero), and every $\F$-inner vector arises in this way.
\end{prop}

\begin{proof}
The proof essentially follows the ideas of \cite[Proposition~3.1]{MR4019469}. 
The vector $h - P_\cM h$ is zero precisely when $h \in \cM$, so let us assume otherwise.

Let us first show that $h - P_\cM h$ is $\F$-inner. Notice that for and $T\in \F \setminus \{I\}$, we have both ${Th \in \cM}$ and $TP_\cM h \in \cM$. In turn, for any $T \in \F \setminus \{I\}$, this implies
\[
h - Ph_\cM \perp T(h - P_\cM h),
\]
and therefore $h - P_\cM h$ is $\mathcal{\F}$-inner. 

Conversely, by the definition of inner, we have $h\perp \cM$, and so $P_\cM h =0$. Hence, $h - P_\cM h = h$, which is inner. 

\end{proof}

In the next section, we will see how an aleph vector also gives a characterization of inner vectors.

\section{Least Squares and Optimal Approximants}\label{least-squares}

This section focuses on algorithmically determining if a vector is cyclic or inner for a given forward operator monoid. 
We first record a well-known result, tailored to our setting. Before doing so, we establish an important convention.

\begin{remark}
For convenience, from here forward, we assume a general relabelling of the monoid $(\Lambda, b, n_0)$ with $\Lambda = \{ n_0, n_0 + 1, n_0 +2, \ldots\}$, but still assume that $n_0$ is the unit for the indexing monoid and that $T_{n_0} = I$.
\end{remark}

\begin{prop}
Let $\F \subseteq \mathcal{B}(\h)$ be a forward operator monoid. TFAE: 
\begin{enumerate}
\item $h \in \h$ is cyclic for $\F$.
\item There exists an $\F$-cyclic vector $\sigma \in [h]_\F$.
\item For every $v \in \h$, there exist scalars $(c_j)_{j = n_0}^\infty \subseteq \C$ so that
\[
\sum_{j=n_0}^\infty c_j T_j h \ \to \ v. 
\]
\end{enumerate}
\end{prop}

\begin{proof}
The implications $(i) \implies (ii)$ and $(iii) \implies (i)$ follow from definition. The implication $(ii) \implies (iii)$ is shown by noting that if $\sigma \in [h]_\F$, then $[\sigma]_\F \subseteq [h]_\F$. 
\end{proof}

This shows that in order to determine if an element $h \in \h$ is $\F$-cyclic, it suffices to approximate a known cyclic vector $\sigma \in \h$ with linear combinations of elements of the set $\{T_jh\}_{j = n_0}^\infty$. This naturally leads us to the following least-squares minimization problem, which asks for \textit{optimal} norm decay in such an approximation: 
\begin{problem}[Optimal approximation]
Let $\F \subseteq \mathcal{B}(\h)$ be a forward operator monoid and let $\sigma \in \h$ be cyclic for $\F$. Suppose $\F = \{T_j\}_{j = n_0}^\infty$ and let $h \in \h$. 
For each $N \ge n_0$,  find the constants solving
\[
\min_{c_{n_0}, \ldots, c_N \in \mathbb{C}} \left\| \sum_{j=n_0}^N c_j T_j h - \sigma \right\|^2.
\]
\end{problem}

Given the Hilbert space structure, we see that this minimization is really just asking us to find the orthogonal projection of $\sigma$ onto the subspace ${V_N := \text{span}\{ T_j h : j = n_0, \ldots, N\}}$ for each $N \ge n_0$. 
We will denote this projection by $P_N \sigma$, and the corresponding coefficients by $c_{n_0}^{(N)}, \ldots, c_N^{(N)}$, i.e., 
\[
P_N \sigma = \sum_{j=n_0}^N c_j^{(N)} T_j h.
\]
We call this projection the $N$th \textit{optimal approximant} to $\sigma$ in $[h]_\F$. We note that when $\F$ is a monoid generated by coordinate shift operators in spaces of analytic functions, these approximants are known as \textit{optimal polynomial approximants}. We point the interested reader to the survey \cite{MR4244844} for more regarding this topic. 

By the orthogonality of the projection, we have 
\[
\sigma - \sum_{j=n_0}^N c_j^{(N)} T_j h \ \perp \sum_{j=n_0}^N b_j T_j h,
\]
for all constants $b_{n_0}, \ldots, b_{N} \in \mathbb{C}$. 
In particular, for each $k = n_0, \ldots N$, we obtain
\[
\left\langle \sigma - \sum_{j=n_0}^N c_j^{(N)} T_j h, T_k h \right\rangle = 0.
\]
Rearranging yields
\[
\sum_{j=n_0}^N c_j^{(N)} \left\langle  T_j h, T_k h \right\rangle = \left\langle \sigma, T_k h \right\rangle.
\]
Putting this together for each $k = n_0, \ldots, N$, we arrive at the following linear system: 
\begin{equation}\label{opt-sys}
\begin{bmatrix}
\|h\|^2 & \langle h, T_{n_0+1} h \rangle & \dots & \langle h, T_N h \rangle \\
\langle T_{n_0+1} h,  h \rangle & \|T_{n_0+1}h \|^2 & \dots & \langle T_{n_0+1} h, T_N h \rangle \\
\vdots & \vdots & \ddots & \vdots\\ 
\langle T_N h,  h \rangle & \langle T_{N}h, T_{n_0+1} h \rangle & \dots & \|T_N h \|^2
\end{bmatrix}
\begin{bmatrix}
c_{n_0}^{(N)} \\
c_{n_0+1}^{(N)}\\
\vdots\\
c_{N}^{(N)}
\end{bmatrix}
=
\begin{bmatrix}
\langle \sigma, h \rangle\\
\langle \sigma, T_{n_0+1}h \rangle\\
\vdots\\
\langle \sigma, T_N h \rangle
\end{bmatrix}.
\end{equation}

Also, again by the orthogonality of the projection, we have 
\[
\| P_N \sigma\|^2 = \langle \sigma, P_N \sigma \rangle.
\]
This tells us that the square distance between $\sigma$ and $V_N$ is
\begin{align*} 
\operatorname{dist}^2(\sigma, V_N) &= \left\| P_N \sigma - \sigma \right\|^2\\
&= \left\| P_N \sigma \right\|^2 - 2\Re{\langle P_N \sigma, \sigma\rangle} + \left\| \sigma \right\|^2\\
&= \left\| \sigma \right\|^2 - \left\| P_N \sigma  \right\|^2. 
\end{align*}
Since $\cup_j V_j$ is dense in $[h]_\F$, we have that $\sigma \in [h]_\F$ if and only if this distance goes to zero. In turn, $h$ is cyclic if and only if 
\[
\left\| P_N \sigma  \right\| \ \xrightarrow{N \to \infty} \ \| \sigma \|.
\]

\begin{remark}
At this juncture, we pause to note several things. First, we have now uncovered that determining cyclicity is equivalent to a linear algebra exercise (finding $P_N\sigma$) and taking a limit ($\lim_{N \to \infty} \|P_N \sigma\|$). However, the element $\sigma$ was taken to be an arbitrary cyclic element. It is natural to ask whether or not the choice of $\sigma$ makes a difference in attempting to detect cyclicity in this way. In particular, is there an $\F$-cyclic element that lends itself to computational advantages when computing optimal approximants? It turns out this is one of the distinct advantages of an aleph vector, which we show now. 
\end{remark}

\begin{theorem}\label{al-approx}
Let $\F \subseteq \mathcal{B}(\h)$ be a forward operator monoid. Let $h, u \in \h$, with $u$ a normalized aleph vector for $\F$. Let $P_N u = \sum_{j=n_0}^N c_j^{(N)} T_j h $ be the orthogonal projection of $u$ onto $V_N: = \operatorname{span}\{T_j h : j = n_0, \ldots, N\}$. Then $h$ is $\F$-cyclic if and only if 
\[
c_{n_0}^{(N)} \langle h,  u \rangle \ \xrightarrow{N \to \infty} \ 1.
\]
\end{theorem}

\begin{proof}
Considering the the previously established distance formula, we have
\begin{align*}
\operatorname{dist}^2(u, V_N) &=  \left\| u \right\|^2 - \left\| P_N u  \right\|^2\\
& = 1 - \left\| P_N u  \right\|^2\\
&= 1- \left\langle \sum_{j=n_0}^N c_j^{(N)} T_j h, u \right\rangle.
\end{align*}
However, as $u$ is an aleph vector, Lemma \ref{unit_ortho} reduces this expression to
\[
\operatorname{dist}^2(u, V_N) = 1 - c_{n_0}^{(N)}\left\langle h, u \right\rangle.
\]
Taking a limit then gives the result.
\end{proof}

\begin{remark}
This result shows that the approximants to an aleph vector should be much easier to compute than the approximants to some other non-aleph cyclic vector; for each $N \ge n_0$, we need only to find the coefficient $c_{n_0}^{(N)}$ instead of the entire sequence $c_{n_0}^{(N)}, \ldots, c_{N}^{(N)}$. When the aleph vector $u$ is approximated, the right hand side of the system \ref{opt-sys} becomes
\[
\left( \langle u, h \rangle, 0, \ldots, 0 \right)^T.
\]
In this case, calculating $\operatorname{dist}^2(u, V_N)$ is equivalent to finding only the $(n_0, n_0)$ entry of the inverse of the Gramian
\[
\left( \langle T_j h, T_k h \rangle \right)_{n_0 \le j,k \le N}. 
\]
This task is, ostensibly, more simple than inverting the entire matrix. 
\end{remark}

Thus far, we have seen that an aleph vector can be used to characterize other cyclic vectors and computationally aid in the determination of cyclicity. It turns out that an aleph can also be used to characterize inner vectors. 
\begin{theorem}\label{opa}
Let $\F \subseteq \mathcal{B}(\h)$ be a forward operator monoid. Let $u, h \in \h$, with $u$ an aleph vector for $\F$. Let $P_N u = \sum_{j=n_0}^N c_j^{(N)} T_j h $ be the orthogonal projection of $u$ onto $\operatorname{span}\{T_j h : j = n_0, \ldots, N\}$. Then $h$ is $\F$-inner if and only if $P_N u = P_{n_0} u$ for all $N \ge n_0$. 
\end{theorem}

\begin{proof}
Consider the system \ref{opt-sys} with respect to $u$ (i.e., replace $\sigma$ with $u$ in Equation \ref{opt-sys}), letting
\[
G_N = (\langle T_j h, T_k h \rangle)_{n_0 \le j,k \le N}.
\]
If $h$ is $\F$-inner, then the first row and column of $G_N$ must be comprised of zeros, except for the $(n_0, n_0)$ entry. The inverse of this matrix must also possess this property.
Thus, we see that $c_j^{(N)} = 0$ for all $j = n_0 +1, \ldots, N$ and $c_{n_0}^{(N)} = c_{n_0}^{(n_0)} $ for all $N \ge n_0$. 

Conversely, suppose that $c_{n_0}^{(N)} = c_{n_0}^{(n_0)}$ for all $N \ge n_0$. Considering the system
\[
\begin{bmatrix}
\|h\|^2 & \langle h, T_{n_0 +1}h \rangle \\
\langle T_{n_0 + 1}h, h \rangle &  \|T_{n_0 + 1}h\|^2
\end{bmatrix}
\begin{bmatrix}
c_{n_0}^{(n_0 + 1)} \\
c_{n_0 + 1}^{(n_0 + 1)}
\end{bmatrix}
=
\begin{bmatrix}
\langle u, h \rangle \\
0
\end{bmatrix}.
\]
quickly yields that $\langle h, T_{n_0 +1}h \rangle = 0$. Again, as $c_{n_0}^{(N)} = c_{n_0}^{(n_0)}$ for all $N \ge n_0$, a simple induction argument then shows that $\langle h, T_k h \rangle = 0$ for all $k > n_0$. 
\end{proof}
Results of this type are known as \textit{stabilization theorems}. If an $\F$-aleph vector exists, then $\F$-inner functions can be characterized by optimal approximants which are all constant multiples of the vector at hand; or, in other words, the optimal approximants of inner vectors are inimical to the  approximation scheme and immediately \textit{stabilize} under the action of the operator monoid. This behavior is, in some sense, antipodal to the behavior of optimal approximants to the aleph in a subspace generated by a cyclic vector; those approximants converge to the aleph vector while the approximants in the case of an inner vector never leave the span of the vector.

\section{Applications and Examples}\label{apps}

\subsection{The PDCP and Nikolskii's Shadow}
In the 1940s, Beurling \cite{MR1057613} and Wintner \cite{MR11497} independently considered the following problem: 
\begin{question}[Periodic Dilation Completeness Problem (PDCP)]
Extend $\psi \in L^2(0,1)$ to an odd 2-periodic function on the real line and consider the integer dilation system $\{ \psi(kx)\}_{k=1}^\infty$. For which $\psi$ is this system dense in $L^2(0,1)$? 
\end{question}
A system $\{ \psi_k\}_{k=1}^\infty$ is called \textit{complete} in a Hilbert space $\h$ if any vector in $\h$ can be approximated arbitrarily well by linear combinations of the functions in the system. In the setting of the PDCP, this is equivalent to $\psi$ being cyclic for the multiplicative forward operator monoid formed by $(B_k\psi)(x) = \psi(kx)$; any such function $\psi$ for which this action turns out a complete system is called a \textit{PDCP function}. This problem is of interest for a number of reasons, including its applications Diophantine approximation, Dirichlet series, and analytic number theory (see, e.g.,  \cite{MR11497, MR4203042}). 

By considering the Fourier series of $\psi$, Beurling connected the PDCP with a problem in Dirichlet series. Further, by using the Bohr lift, Beurling connected the PDCP to a problem concerning a power series in infinitely many variables. 

Nikolskii \cite{Nik} translated that problem to the space
\[
H^2_0 := \{ f \in H^2 : f(0) = 0\}
\]
by considering a multiplicate forward operator monoid, which, for $n\ge1$ is given by
\[
(T_nf)(z) = f(z^n), \ \ f\in H^2_0.
\]

There is a bijective correspondence between cyclic functions for this operator monoid, PDCP functions, and cyclic vectors for the monoid generated by the coordinate forward shifts acting on the Hardy space of the infinite-dimensional polydisc (see, e.g., \cite{MR4203042} and references therein). 

We will consider the power dilation monoid mentioned above and recover a result of Nikolskii by using the aleph vector for this monoid.

\begin{prop}
Consider the forward operator monoid $\mathcal{T} = (T_{n})_{n \ge 1}$ acting on $H^2_0$, where 
\[
(T_nf)(z) = f(z^n).
\]
Up to a constant multiple, $\mathcal{T}$ has a unique aleph, which is the identity function $u(z) = z$.
\end{prop}
\begin{proof}
It is immediate that $(T_{n}u)(z) = z^n$, so $[u]_{\mathcal{C}}$ is just the closed linear span of the non-constant polynomials, which is equal to $H^2_0$. Hence, $u$ is $\mathcal{T}$-cyclic. 
Further, 
$\langle u, T_{n}u \rangle = \langle z, z^n \rangle = \delta_{n,1}$, so $u$ is $\mathcal{T}$-inner. 
The result then follows from Theorem \ref{uniqueness}.
\end{proof}

We can recover, with our framework, an example mentioned by Nikolskii \cite{Nik}. 

\begin{proposition}
Consider the forward operator monoid $\mathcal{T} = (T_{n})_{n \ge 1}$ acting on $H^2_0$, where 
\[
(T_nf)(z) = f(z^n).
\]
For any $d\ge1$ and $\lambda \in \C$, the function $f_\lambda(z) = z(\lambda - z^d)$ is cyclic for $\mathcal{T}$ if and only if $|\lambda| \ge 1$. 
\end{proposition}

\begin{proof}
Suppose $g \in [f_\lambda]_\mathcal{T}^\perp$ and put $g(z) = \sum_{k \ge 1} g_k z^k$. For any $j\ge 1$, we have
\begin{align*}
0 & = \langle g, T_j f_\lambda \rangle\\
&= \langle g, \lambda z^j - z^{(d+1)j} \rangle \\
& = \overline{\lambda} g_j - g_{(d+1)j}.
\end{align*}
One may check that the solution to this recurrence relation is given by 
\[
g_j = g_1 \overline{\lambda}^{\frac{\log(j)}{\log(d+1)} - 1}
\]
However, 
\[
\| g \|^2 = |g_1|^2 \sum_{j \ge 1} \left|\lambda\right|^{\frac{\log(j^2)}{\log(d+1)} - 2}.
\]
If $|\lambda| \ge 1$, then this series converges if and only if $g_1 = 0$, and equivalently, if and only if $g \equiv 0$. The result then follows. 
\end{proof}

Let us move to inner functions in this setting. In \cite{Nik}, Nikolskii characterized functions in the eigenvector bundle of $(T_n^*)$, which can be used to produce some examples of $\mathcal{T}$-inner functions (e.g., those functions for which $T_n^*f = \delta_{1,n}f$ are all $\mathcal{T}$-inner, but the set of inner functions is much larger than this). Let us give a simple sufficient condition for a function to be $\mathcal{T}$-inner.

\begin{prop}
Consider the forward operator monoid $\mathcal{T} = (T_{n})_{n \ge 1}$ acting on $H^2_0$, where 
\[
(T_nf)(z) = f(z^n).
\]
Suppose $\{p_j\}_{j\ge 1} \subseteq \mathbb{Z}_{>1}$ with $\gcd(p_j, p_k) = 1$ for all $j$ and $k$ whenever $j \neq k$. Then for any choice of constants $\{ a_j\}_{j=1}^\infty \subseteq \C$, the function
\[
f(z) = \sum_{j \ge 1} a_j z^{p_j}
\]
is $\mathcal{T}$-inner. 
\end{prop}

\begin{proof}
For $n \ge 1$, put $S_n =  \supp\widehat{T_n f}$ and notice that since $\gcd(p_j, p_k) = 1$ for each $j \neq k$, we have $S_1 \cap S_n = \emptyset$ for all $n > 1$. In turn, $\langle f, T_n f \rangle = 0$ for all $n >1$ and the result follows. 
\end{proof}

We now move to another related forward operator monoid.

\subsection{Noor's Monoid}
A forward operator monoid related to the previous example is $\mathcal{W} = (W_n)_{n \ge 1}$, considered on $H^2$, given by 
\[
(W_n f)(z) = \frac{1-z^n}{1-z}f(z^n).
\]
This monoid is multiplicative; indeed, one may check 
\[
W_mW_n = W_{mn} \ \ \ \forall n,m \ge 1.
\]
Standing on the shoulders of Nyman \cite{Nyman}, Beurling \cite{Beurling-RH}, and Ba\'ez-Duarte \cite{B-D}, Noor has shown the following interesting result:
\begin{theorem}[Noor, \cite{Noor}]\label{SWNoor}
For $k \ge 2$, let 
\[
h_k(z) = \frac{1}{1-z}\log\left( \frac{1 - z^k}{k(1 - z)}\right).
\]
Let $\mathcal{N} := \operatorname{span}\{h_k : k \ge 2 \}$ and let $\mathcal{W} = \{ W_n\}_{n \ge 1}$, where $W_n$ acts on $f \in H^2$ by 
\[
(W_n f)(z) = \frac{1 - z^n}{1-z} \ f(z^n).
\]
The following are equivalent:
\begin{enumerate}
\item The closure of $\mathcal{N}$ in $H^2$ contains a cyclic vector for $\mathcal{W}$. 
\item The vectors $\{h_k \}_{k \ge 2}$ are complete in $H^2$.
\item The Riemann Hypothesis is true.
\end{enumerate}
\end{theorem}

As pointed out in \cite{MR4567499}, the power dilation $T_n$ in the previous subsection and the operator $W_n$ are semiconjugate in the sense that, for all $n\ge1$, we have
\[
T_n(I - S) = (I - S)W_n.
\]
Also pointed out there is that the unit constant function is $\mathcal{W}$-cyclic. However, for any $n\ge 1$, 
\[
\langle 1, W_n1 \rangle = \langle 1, 1 + z + \dots + z^{n-1} \rangle = 1,
\]
so 1 is not $\mathcal{W}$-inner, and thus is not an aleph for $\mathcal{W}$. 

Let us now uncover the aleph for $\mathcal{W}$. 
Using this, we will give another reformulation of Noor's reformulation of the Riemann Hypothesis. 

\begin{proposition}\label{W-aleph}
Let $\mathcal{W} = \{ W_n\}_{n\ge 1}$ be the forward operator monoid given by
\[
(W_n f)(z) = \frac{1-z^n}{1-z}f(z^n),
\]
considered on $H^2$. Up to a constant multiple, there is a unique aleph vector for $\mathcal{W}$, which is the function $u(z) = 1 - z$.
\end{proposition}
\begin{proof}
First we will show that $u$ is inner; for $k \ge 2$, we have 
\begin{align*}
\langle u, W_k u \rangle & = \langle 1-z, (1 + z + \dots + z^{k-1})(1-z^k) \rangle \\
& = \langle 1-z, 1 + z + \dots + z^{k-1} - z^k - \dots - z^{2k-1} \rangle \\
&= 1 - 1 = 0.
\end{align*}
For cyclicity, put $g(z) = \sum_{n\ge 0} g_n z^n$ and suppose that $g \in [u]_{\mathcal{W}}^\perp$. By studying Maclaurin coefficients, we will show that $g \equiv 0$.
Observe, for any $k\ge 1$, we have
\begin{align*}
0 &= \left\langle \sum_{n\ge 0} g_n z^n, W_k(1-z) \right\rangle\\
&= g_0 + \dots + g_{k-1} - g_k - \dots - g_{2k-1}.
\end{align*}
In turn,
\begin{align*}
0 &= \left\langle \sum_{n\ge 0} g_n z^n, (W_{k+1} - W_k)(1-z) \right\rangle\\
&= 2g_k - g_{2k} - g_{2k+1},
\end{align*}
which implies 
\begin{equation}\label{avg}
g_k = \frac{g_{2k} + g_{2k+1}}{2}.
\end{equation}
Varying $k \ge 1$ and using Equation \ref{avg}, we iteratively determine  expressions for $g_0$ as:
\begin{align*}
g_0 &= g_1\\
& = \frac{g_{2} + g_{3}}{2}\\
& = \frac{g_{4} + g_5 + g_{6} + g_{7}}{4}\\
& \ \ \ \vdots\\
& = \frac{1}{2^k} \sum_{k = 2^j}^{2^{j+1} - 1} g_j.
\end{align*}
This is a dyadic average; the average property implies that for each $k \ge 1$, there must be some $g_{n_k}^*$, $2^k \le n_k \le 2^{k+1} - 1$,  so that $|g_{n_k}^*| \ge |g_0|$. With this in hand, we break into cases: 
\begin{enumerate}
\item[I.] $g_0 \neq 0$: In this case, WLOG, assume $g_0 = 1$. Then, by the dyadic expression above, we have
\[
\|g\|_{H^2}^2 = \sum_{k\ge0} |g_k|^2 \ge \sum_{k\ge0} |g_{n_k}^*|^2 =  \infty.
\]
However, this is absurd, and we end up with the following:
\item[II.] $g_0  = 0$: In this case, assume WLOG, that $g_2 = 1$. We then set sail on the same ship as in case I to conclude that $g_2 = 0$. 
\end{enumerate}
Ultimately, our voyage arrives at the fact that $g_k = 0$ for all $k \ge 0$. Thus, $g \equiv 0$, and we conclude that $[u]_{\mathcal{W}}^\perp = \{0\}$, which means that $u$ is $\mathcal{W}$-cyclic. 
\end{proof}

As an immediate corollary, we recover a result in \cite{MR4567499}. 
\begin{corollary}
Let $\mathcal{W} = \{W_n\}_{n\ge1}$ be the forward operator monoid given by 
\[
(W_n f)(z) = \frac{1-z^n}{1-z}f(z^n),
\]
considered on $H^2$.
If $f \in H^2$ satisfies $f(0) = f'(0)$, then 
\[
\operatorname{dist}^2(1 - z, [f]_\mathcal{W}) = 2.
\]
In turn, any such $f$ cannot be $\mathcal{W}$-cyclic. 
\end{corollary}

\begin{proof}
In this setting, the linear system \ref{opt-sys}, with respect to $\mathcal{W}$ and the aleph $u(z) = 1-z$, yields (for any $N \ge 1$)
\begin{align*}
\left(\langle W_j f, W_k f \rangle\right)_{1 \le j, k \le N} \left(c_0^{(N)}, \ldots, c_N^{(N)} \right)^T
&= \left( \langle u , f\rangle , 0, \ldots, 0\right)^T\\
&= \left( \overline{f(0) - f'(0)} , 0, \ldots, 0\right)^T .
\end{align*}
If $f(0) = f'(0)$ then $u \perp f$ and so the orthogonal projection of $u$ onto $V_N:= \spn\{W_k f : k = 1, \ldots, N\}$ is identically zero. The corresponding distance formula gives
\[
\operatorname{dist}^2\left(u, V_N\right)= \|u\|^2_{H^2} = 2. 
\]
Taking a limit as $N \to \infty$ then gives the result.
\end{proof}

Using the aleph for $\mathcal{W}$, we provide a reformulation of the Riemann Hypothesis.

\begin{theorem}\label{RH-reform}
Let $\mathcal{W} = \{W_n\}_{n\ge1}$ be the forward operator monoid given by 
\[
(W_n f)(z) = \frac{1-z^n}{1-z}f(z^n),
\]
considered on $H^2$, let $\mathcal{N}$ be as in Theorem \ref{SWNoor}, and let $\cls_{H^2}(\mathcal{N})$ denote the closure of $\mathcal{N}$ in $H^2$.

Then the Riemann Hypothesis is true if and only if there exists a function $f \in \cls_{H^2}(\mathcal{N})$ so that 
\[
c_1^{(N)} \xrightarrow{N \to \infty} \frac{\sqrt{2}}{\overline{f(0) - f'(0)}},
\]
where $c_1^{(N)}$ arises from the solution to the linear system
\[
\left( \langle W_jf, W_kf\rangle_{H^2}\right)_{1 \le j,k \le N} \left(c^{(N)}_1, \ldots, c^{(N)}_N\right)^T = \frac{1}{\sqrt{2}}\left( \overline{f(0) - f'(0)}, 0, \ldots, 0 \right)^T.
\]
\end{theorem}

\begin{proof}
Let $h_1 \equiv 0$ and let $h_k$ ($k \ge 2$) and $\mathcal{N}$ be as in Theorem \ref{SWNoor}. As pointed out in \cite{Noor}, putting 
\[
h_k(z) = \frac{1}{1-z}\left[ \log(1-z^k) - \log(1-z) - \log(k)\right],
\]
it is an elementary exercise to verify that 
\[
W_j h_k = h_{jk} - h_j
\]
for all $j,k \ge 1$. In turn, $\mathcal{W}(\mathcal{N}) \subseteq \mathcal{N}$, and since $\mathcal{W}(\mathcal{N}) \supseteq W_1(\mathcal{N}) = \mathcal{N}$, we have $\mathcal{W}(\mathcal{N}) = \mathcal{N}$. 

Via Theorem \ref{SWNoor}, we have that the Riemann Hypothesis is equivalent to the existence of a $\mathcal{W}$-cyclic vector, say $f$, being an element of $\cls_{H^2}(\mathcal{N})$. This is equivalent to an aleph vector for $\mathcal{W}$ being an element of $[f]_\mathcal{W}$. Invoking Theorem \ref{al-approx} with the normalized aleph vector $(1 - z)/\sqrt{2}$ then gives the result. 
\end{proof}

Let us conclude with a few remarks. The first difficulty with using Theorem \ref{RH-reform} is that even for the most simple functions $f \in \cls_{H^2}(\mathcal{N})$, the inner products $\langle W_jf, W_kf\rangle_{H^2}$ are non-trivial to compute. Should this difficulty be overcome, the next trouble here comes in inverting the corresponding Gram matrix, which is also a difficult task. However, there is some additional structure to these matrices. As pointed out in \cite{MR4567499}, for all $n\ge 1$, 
\[
W_n^*W_n = nI,
\]
and so there is a reduction in the number of degrees of freedom in the Gram matrix $\left(\langle W_jf, W_kf\rangle_{H^2}\right)_{1\le j,k \le N}$. However, it is presently unclear how to take advantage of this reduction.

\subsection*{Acknowledgements} 
Special thanks to Prof. Hari Bercovici for helpful discussion.

\bibliography{refs}
\bibliographystyle{amsplain}
\end{document}